\documentclass[12pt, reqno]{amsart}
\usepackage{amsmath, amsthm, amscd, amsfonts, amssymb, graphicx, color}

\newtheorem{theorem}{Theorem}[section]
\newtheorem{lemma}[theorem]{Lemma}

\newtheorem{corollary}[theorem]{Corollary}
\theoremstyle{definition}
\newtheorem{definition}[theorem]{Definition}
\newtheorem{example}[theorem]{Example}

\theoremstyle{remark}
\newtheorem{remark}[theorem]{Remark}
\begin{document}

\setcounter{page}{1}

\title{ON $e^*$-TOPOLOGICAL RINGS}

\author{C. DALKIRAN$^{\rm a}$,  M. ÖZKOÇ$^{\rm b,\ast}$}

\address{$^{B}$Department of Mathematics, Muğla Sıtkı Koçman University, 48000, Menteşe-Muğla, Turkey.}

\email{murad.ozkoc@mu.edu.tr}

\address{$^{A}$Department of Mathematics, Graduate School of Natural and Applied Sciences, Muğla Sıtkı Koçman University, 48000, Menteşe-Muğla, Turkey.}

\email{cndlkrn10@gmail.com}


\subjclass[2010]{54H13}

\keywords{topological ring, $e^*$-open, $e^*$-topological ring, $e^*$-continuous.}

\date{--- $^{*}$Corresponding author}

\begin{abstract}

The main purpose of this paper is to introduce the concept of $e^*$-topological ring. This class appears as a generalized form of the class of $\beta$-topological rings. In addition, we have discussed the relation between the concept of $e^*$-topological ring and some other types of topological rings existing in the literature. Also, some fundamental results about $e^*$-topological rings are revealed. Furthermore, we give some counterexamples related to our results.
 \\
\end{abstract}
\maketitle

\section[Introduction]{Introduction}
To find solutions to some problems in topology, it is sometimes necessary to use algebra. This situation leads to the definition of different concepts in the related field. One of these concepts is the concept of topological ring. In order to better understand topological rings, the concept of topological group should be well known. A topological group is a group $G$ that is also a topological space such that the addition and the inversion are continuous as maps $\psi: G\to G, x\mapsto -x$ and $\varphi : G\times G\to G, (x,y)\mapsto x+y,$ where $G\times G$ carries the product topology. The concept of topological ring has been first introduced in  \cite{Kap1,Kap2} by Kaplansky. A topological ring is a ring $R$
that is also a topological space such that both the addition and the multiplication are continuous as maps $\varphi : R\times R\to R,$ where $R\times R$ carries the product topology. That means $R$ is an additive topological group and a multiplicative topological semigroup. 
The types of open sets in the literature, such as $\alpha$-open \cite{Njastad}, semi-open \cite{Levine},  pre-open \cite{Mas}, $\beta$-open \cite{Abd}, etc. allow a generalization of the notion of topological ring. Studying the properties of these generalized forms and investigating their relations with topological rings are just some of the different advances in the literature. Some of the recent advancements in this direction are $\beta$-topological rings \cite{Bil}, irresolute topological rings \cite{Salih} and $\alpha$-irresolute topological rings \cite{M.Ram}.
In 2021, Billawaria et al. have been studied $\beta$-topological ring which is more general notion than the notion of topological ring \cite{Bil}. They have revealed some fundamental properties of $\beta$-topological rings. 

In this paper, we introduce the notion of $e^*$-topological ring by utilizing $e^*$-open sets defined by Ekici in \cite{Ekici}. We obtain some of its fundamental properties. Also, we compare between this notion and some notions existing in the literature. Furthermore, we give some counterexamples regarding our results obtained in the scope of this study.



\section[Preliminaries]{Preliminaries}
Throughout this paper, $(X,\tau )$ and $(Y,\sigma)$ (or simply $X$ and $Y$) always mean topological spaces on which no separation axioms are assumed unless explicitly stated. For a subset $A$ of a topological space $X,$ the closure of $A$ and the interior of $A$ are denoted by $cl(A)$ and $int(A)$, respectively. The family of all closed (resp. open) sets of $X$ is denoted by $C(X)(\text{resp. } O(X))$. In addition, the family of all open sets of $X$ containing a point $x$ of $X$ is denoted by $O(X,x)$.  Recall that a subset $A$ of a space $X$ is called regular open \cite{sto} (resp. regular closed \cite{sto}) if $A = int(cl(A))$ (resp. $A = cl(int(A)))$. The family of all regular open subsets of $X$ is denoted by $RO(X).$ The family of all regular open sets of $X$ containing a point $x$ of $X$ is denoted by $RO(X,x).$
A subset $A$ of a space $X$ is called $\delta$-open \cite{vel} if for each $x\in A$ there exists a regular open set $V$ such that $x\in V\subseteq A$. A set $A$ is said to be $\delta$-closed if its complement is $\delta$-open. The intersection of all regular closed sets of $X$ containing $A$ is called the $\delta$-closure \cite{vel} of $A$ and is denoted by $\delta$-$cl(A)$. Dually, the union of all regular open sets of $X$ contained in $A$ is called the $\delta$-interior \cite{vel} of $A$ and is denoted by $\delta$-$int(A)$. 
A subset $A$ of a space $X$ is called $e^*$-open  
if $A\subseteq cl(int(\delta\text{-}cl(A))).$ The complement of an $e^*$-open  set is called $e^*$-closed. The intersection of all $e^*$-closed  sets of $X$ containing
$A$ is called the $e^*$-closure of $A$ and is denoted by $e^*$-$cl(A)$. Dually, the union of all $e^*$-open sets of $X$ contained in $A$ is called the $e^*$-interior of $A$ and is denoted by $e^*$-$int(A)$. The family of all $e^*$-open subsets (resp. $e^*\text{-}$closed) $X$ denoted by $e^*O(X)$ (resp. $e^*C(X))$. The family  of all $e^*\text{-}$open (resp. $e^*$-closed) sets of $X$ containing a point $x$ of $X$ denoted by $e^*O(X,x)$ (resp. $e^*C(X,x)).$
\begin{definition}\cite{Kap1}
 Let $(R,+,\cdot)$ be a ring and $\tau$ be a topology on $R.$ The quadruple $(R,+,\cdot,\tau)$ is called a topological ring if the following conditions are satisfied:
 
 $i)$ For each $x,y\in R$ and each open set $W\in O(R,x+y),$ there exist open sets $U$ and $V$ in $R$ containing  $x$ and $y$, respectively, such that $U+V\subseteq W,$ 

 $ii)$ For each $x\in R$ and each open set $V\in O(R,-x)$,  there exists $U\in O(R,x)$ such that $-U\subseteq  V,$
 
 $iii)$ For each $x,y\in R$ and each open set  $W\in O(R,xy)$, there exist  open sets $U$ and $V$ in $R$ containing  $x$ and $y$, respectively, such that $UV\subseteq W.$ 
\end{definition}

  \begin{definition}\cite{Bil}
 Let $(R,+,\cdot)$ be a ring and $\tau$ be a topology on $R.$ The quadruple $(R,+,\cdot,\tau)$ is called an $\beta\text{-}$topological ring if the following conditions are satisfied:
 
 $i)$ For each $x,y\in R$ and each open set $W\in O(R,x+y),$ there exist $\beta$-open sets $U$ and $V$ in $R$ containing  $x$ and $y$, respectively, such that $U+V\subseteq W,$ 

 $ii)$ For each $x\in R$ and each open set $V\in O(R,-x)$,  there exists $U\in\beta O(R,x)$ such that $-U\subseteq  V,$
 
 $iii)$ For each $x,y\in R$ and each open set  $W\in O(R,xy)$, there exist  $\beta$-open sets $U$ and $V$ in $R$ containing  $x$ and $y$, respectively, such that $UV\subseteq W.$   
 \end{definition}  

\begin{definition}
\cite{Ekici} A function $f:(X,\tau)\to (Y,\sigma)$ is said to be $e^*$-continuous if $f^{-1}[A]$ is $e^*$-open in $X$ for every $A\in\sigma.$   
\end{definition}

\begin{lemma}
\cite{Ekici} A function $f:(X,\tau)\to (Y,\sigma)$ is $e^*$-continuous if and only if for every $x\in X$ and for every $V\in O(Y,f(x)),$ there exists $U\in e^*O(X,x)$ such that $f[U]\subseteq V.$    
\end{lemma}
 
\section{$e^*$-Topological Rings}
In this section, we introduce and study the concept of $e^*$-topological ring by utilizing $e^*$-open sets.

 
 \begin{definition}\label{def31}
Let $(R,+,\cdot)$ be a ring and $\tau$ be a topology on $R.$ The quadruple $(R,+,\cdot,\tau)$ is called an $e^*$-topological ring if the following conditions are satisfied:
 
$i)$ For each $x,y\in R$ and each open set $W\in O(R,x+y),$ there exist $U\in e^*O(R,x)$ and $V\in e^*O(R,y)$ such that $U+V\subseteq W,$ 

$ii)$ For each $x\in R$ and each open set $V\in O(R,-x)$, there exists $U\in e^*O(R,x)$ such that $-U\subseteq  V,$
 
$iii)$ For each $x,y\in R$ and each open set  $W\in O(R,xy)$, there exist  $U\in e^*O(R,x)$ and $V\in e^*O(R,y)$ such that $UV\subseteq W.$    
 \end{definition}
 
\begin{remark}
Since every $\beta$-open set is an $e^*$-open set, it is not difficult to see that every $\beta$-topological ring is an $e^*$-topological ring. However, the converse need not always be true as shown by the following example. 
\end{remark}	

\begin{example} \label{34}
Let $R=\{a,b,c,d\}$ and  $\tau=\{\emptyset,R,\{a\},\{a,b\}\}.$ Let the addition and multiplication operations on $R$ be as given in the tables below:

\[
\begin{tabular}{|c|c|c|c|c|}
\hline 
$+$ & $a$ & $b$ & $c$ & $d$ \\
\hline
$a$ & $a$ & $b$ & $c$ & $d$ \\
\hline
$b$ & $b$ & $c$ & $d$ & $a$ \\
\hline
$c$ & $c$ & $d$ & $a$ & $b$ \\
\hline
$d$ & $d$ & $a$ & $b$ & $c$ \\
\hline
\end{tabular}
\hspace{2cm}
\begin{tabular}{|c|c|c|c|c|}
\hline 
$\cdot$ & $a$ & $b$ & $c$ & $d$ \\
\hline
$a$ & $a$ & $a$ & $a$ & $a$ \\
\hline
$b$ & $a$ & $c$ & $a$ & $c$ \\
\hline
$c$ & $a$ & $a$ & $a$ & $a$ \\
\hline
$d$ & $a$ & $c$ & $a$ & $c$ \\
\hline
\end{tabular}
\]
\end{example}
In this topological space, simple calculations show that $e^*O(R)=2^{R}$ and $\beta O(R)=\{\emptyset,R,\{a\},\{a,b\},\{a,c\},\{a,d\},\{a,b,c\},\{a,b,d\},\{a,c,d\}\}.$ Then, it is clear that $(R,+,\cdot,\tau)$ is an $e^*$-topological ring but it is not a $\beta$-topological ring.
 
 \begin{example} Let $(\mathbb{R},+,\cdot)$ be the ring of real numbers with usual topology $\mathcal{U}$. Then, $(\mathbb{R},+,\cdot,\mathcal{U})$ is an $e^*\text{-}$topological ring.
\end{example}

 \begin{example} Let  $(R,+,\cdot)$ be any ring with the discrete topology $2^X.$ Then, $(R,+,\cdot,2^X)$ is an $e^*\text{-}$topological ring.
 \end{example}

\begin{theorem}\label{t.3.4} Let $(R,+,\cdot,\tau)$ be an $e^*\text{-}$topological ring. Then, the following properties hold.

$a)$ If $A\in O(R),$ then $-A\in e^*O(R),$

$b)$ If $A\in O(R)$ and $x\in R$, then $x+A\in e^*O(R).$ 
\end{theorem}

\begin{proof} $a)$ Let $A\in O(R)$ and $x\in R.$ 
\\
$A\in O(R)\Rightarrow -A\subseteq R\Rightarrow e^*\text{-}int(-A)\subseteq -A\ldots (1)$

Now, let $y\in -A.$ Our aim is to show that $y\in e^* \text{-}int(-A).$
\\
$\left.\begin{array}{rr} y\in -A \Rightarrow   -y\in A \\  A\in O(R) \end{array}\right\}\overset{\text{Definition } \ref{def31}}\Rightarrow (\exists U\in e^* O(R,y)) (-U\subseteq A)
$
\\
$\begin{array}{l}\Rightarrow (\exists U\in e^* O(R,y)) (U\subseteq -A)\end{array}$\\
$\begin{array}{l}\Rightarrow y\in e^* \text{-}int(-A)\end{array}$

Then, we have $-A\subseteq e^*\text{-}int(-A)\ldots (2)$

$(1),(2)\Rightarrow e^*\text{-}int(-A)=-A.$
\\

$b)$ Let $A\in O(R)$ and $r\in R.$ Our aim is to show that $x+A\in e^*O(R).$ For this, we will prove tha $x+A=e^*\text{-}int(x+A).$ Now, let $y\in x+A$. If we prove $y\in e^*\text{-}int(x+A),$ then the proof complete.
\\
$\left.\begin{array}{rr} y\in x+A  \Rightarrow (\exists z\in A)(y=x+z)  \\ A\in O(R)\end{array}\right\}\Rightarrow -x+y\in A\in O(R)$
\\
$\begin{array}{l} \overset{\text{Definition }\ref{def31}}{\Rightarrow}  (\exists U\in e^*O(R,-x)) (\exists V\in e^*O(R,y))(- x+V\subseteq U+V\subseteq A)
\end{array}
$
\\
$
\begin{array}{l} \Rightarrow  (\exists V\in e^*O(R,y))(-x+V\subseteq A) \end{array}
$
\\
$
\begin{array}{l} \Rightarrow  (\exists V\in e^*O(R,y))(V\subseteq x+A)) \end{array}
$
\\
$
\begin{array}{l} \Rightarrow  y\in e^*\text{-}int(x+A).
\end{array}
$
\end{proof}

\begin{corollary}Let $(R,+,\cdot,\tau)$ be an $e^*\text{-}$topological ring and $A\subseteq R.$ Then, the following statements hold.

$a)$ If $A\in O(R),$ then $-A\subseteq cl(int(\delta\text{-} cl(-A))),$

$b)$ If $A\in O(R),$ then $x+A\subseteq cl(int(\delta\text{-}cl(x+A)))$ for every $x\in R.$
\end{corollary}

\begin{theorem}
Let $(R,+,\cdot,\tau)$ be an $e^*\text{-}$topological ring and $A\subseteq R.$ Then, the following properties hold. 

$a)$ If $A\in C(R),$ then $-A\in e^*C(R),$ 

$b)$ If $x\in R$ and $A\in C(R)$, then $x+A\in e^*C(R).$ 
\end{theorem}
\begin{proof} $a)$ Let $A\in C(R).$ Our aim is to show that $-A\in e^*C(R).$ Now, let $y\in e^*\text{-}cl(-A).$ We will show that $y\in -A,$ namely, $-y\in A.$ Let $W\in O(R,-y)$.
\\
$\left.\begin{array}{rr} W\in O(R,-y) \Rightarrow (\exists U\in e^*O(R,y))(-U\subseteq W)\\ y\in e^*\text{-}cl(-A)\end{array}\right\}\Rightarrow (U\subseteq -W)(U\cap (-A)\ne\emptyset)
$
\\
$
\begin{array}{l}
\Rightarrow \emptyset \neq U\cap (-A)\subseteq (-W)\cap (-A)
\end{array}
$
\\
$
\begin{array}{l}
\Rightarrow  W\cap A\ne\emptyset
\end{array}
$

Then, we have $-y\in cl(A).$ Since $A\in C(R),$ we get $-y\in A,$ namely, $y\in -A.$ Thus, we have $-A\subseteq e^*\text{-}cl(-A)\subseteq -A,$ namely, $-A=e^*\text{-}cl(-A).$ This means $-A\in e^*C(R).$
\\

$b)$ Let $x\in R$ and $A\in C(R).$ Our aim is to show that $x+A\in e^*C(R).$ Now, let $y\in e^*\text{-}cl(x+A).$ We will prove that $y\in x+A,$ namely, $-x+y\in A.$ Let $W\in O(R,-x+y).$ 
\\
$
\left.\begin{array}{rr} W\in O(R,-x+y) \Rightarrow (\exists U\in e^*O(R,-x))(\exists V\in e^*O(R,y)) (U+V\subseteq W)\\ y\in e^*\text{-}cl(x+A)\end{array}\right\}\Rightarrow
$
\\
$
\begin{array}{l}\Rightarrow (U+V\subseteq W)(V\cap (x+A)\ne\emptyset)
\end{array}
$
\\
$
\begin{array}{l}\Rightarrow \emptyset\ne (-x+V)\cap A \subseteq (U+V)\cap A\subseteq W\cap A
\end{array}
$
\\
$
\begin{array}{l}\Rightarrow W\cap A\neq\emptyset
\end{array}
$

Then, we have $-x+y\in cl(A).$ Since $A\in C(R),$ we get $-x+y\in A.$  Hence, $y\in x+A.$ 
\end{proof}

\begin{corollary}
Let $(R,+,\cdot,\tau)$ be an $e^*\text{-}$topological ring and $A\subseteq R.$ Then, the following statements hold.

$a)$ If $A\in C(R),$ then $int(cl(int(-A)))\subseteq -A,$

$b)$ If $A\in C(R),$ then $int(cl(int(x+A)))\subseteq x+A$ for all $x\in R$.
\end{corollary}

\section{MAIN RESULTS}	
In this section, we obtain some fundamental properties of $e^*\text{-}$topological ring. In addition, this section contains the definition of $e^*\text{-}$topological rings with unity and some fundamental results on it.\\

\begin{theorem}Let $(R,+,\cdot,\tau)$ be an $e^*\text{-}$topological ring. Then, the following functions are $e^*\text{-}$continuous:

$a)$ For a fixed $x\in R$, $f_x: R\to R$ defined by $f_x (y)=x+y$ for all $y\in R,$

$b)$ $f: R\to R$ defined by $f(x)=-x$ for all $x\in R$.
\end{theorem}

\begin{proof} $a)$ $V\in O(R).$ Our aim is to show that $f^{-1}[V]\in e^*O(R).$ 
\\
$\left.\begin{array}{rr} f^{-1}[V]=\{y\in R|f_x(y)\in V\}=\{y\in R|y\in -x+V\} = -x+V \\ (x\in R)(V\in O(R))\end{array}\right\}\overset{\text{Theorem }\ref{t.3.4}}{\Rightarrow} $
\\
$\begin{array}{l} \Rightarrow f^{-1}[V]\in e^*O(R).\end{array}$

$b)$ $V\in O(R).$ Our aim is to show that $f^{-1}[V]\in e^*O(R).$ 
\\
$\left.\begin{array}{rr} f^{-1}[V]=\{x\in R|f(x)\in V\}=\{x\in R|-x\in V\} = -V \\ V\in O(R)\end{array}\right\}\overset{\text{Theorem }\ref{t.3.4}}{\Rightarrow} $
\\
$\begin{array}{l} \Rightarrow f^{-1}[V]\in e^*O(R).\end{array}$
\end{proof}

\begin{definition} Let $(R,+,\cdot,\tau)$ be an $e^*\text{-}$topological ring. If $(R,+,\cdot)$ itself is a ring with unity, then $(R,+,\cdot,\tau)$ is said to be an $e^*\text{-}$topological ring with unity. Also, we will use the standard notation $R^*$ for the set of all invertible elements in $(R,+,\cdot).$
\end{definition}

\begin{theorem}
Let $(R,+,\cdot,\tau)$ be an $e^*$-topological ring with unity and $A\subseteq R.$ Then, the following properties hold.

$a)$ If $A\in O(R),$ then $Ar$ is $e^*$-open in $R$ for each $r\in R^*,$

$b)$ If $A\in O(R),$ then $rA$ is $e^*$-open in $R$ for each $r\in R^*.$
\end{theorem}
\begin{proof} $a)$ Let $A\in O(R)$ and $r\in R^*.$ We will prove $Ar\in e^*O(R).$ If we prove $Ar\subseteq e^*\text{-}int(Ar),$ then the proof complete. Let $y\in Ar.$
\\
$\left.\begin{array}{rr}y\in Ar\Rightarrow (\exists k\in A)(y=kr) \\ r\in R^*\end{array} \right\} \!\!\!\!\! \begin{array}{rr} \\ \left. \begin{array}{rr} \Rightarrow (\exists k\in A)(yr^{-1}=k) \\ A\in O(R) \end{array} \right\} \Rightarrow \end{array}$
\\
$
\begin{array}{l} 
\Rightarrow  (\exists U\in e^*O(R,y))(\exists V\in e^*O(R,r^{-1}))(Ur^{-1}\subseteq UV\subseteq A)
\end{array}
$
\\
$
\begin{array}{l} 
\Rightarrow (\exists U\in e^*O(R,y))(U\subseteq Ar) 
\end{array}
$
\\
$
\begin{array}{l} 
\Rightarrow y\in e^*\text{-}int(Ar) 
\end{array}
$

Then, we have $Ar\subseteq e^*\text{-}int(Ar)\subseteq Ar$ which means $Ar\in e^*O(R).$ 
\\

$b)$ It is proved similarly to $(a).$
\end{proof}

\begin{theorem}
Let $(R,+,\cdot,\tau)$ be an $e^*\text{-}$topological ring with unity and $A\subseteq R.$ Then, the following properties hold.

$a)$ If $A\in C(R),$ then $Ar\in e^*C(R)$ for each $r\in R^*,$ 

$b)$ If $A\in C(R),$ then $rA\in e^*C(R)$ for each $r\in R^*.$
\end{theorem}
\begin{proof} Let $A\in C(R)$ and $r\in R^*.$\\
$\left.\begin{array}{rr} y\notin rA \Rightarrow   (\forall k\in A)(y\ne rk)\\ \text{Hypothesis}\end{array}\right\}\Rightarrow (\forall k\in cl(A))(y\ne rA)
$
\\
$\begin{array}{l}\Rightarrow (\forall k\in cl(A))(r^{-1}y\ne k)
\end{array}$
\\
$\begin{array}{l}\Rightarrow r^{-1}y\notin cl(A)
\end{array}$
\\
$\begin{array}{l}\Rightarrow (\exists U\in O(R,r^{-1}y))(U\cap A=\emptyset)
\end{array}$
\\
$\begin{array}{l}\Rightarrow (\exists K\in e^*O(R,r^{-1}))(\exists M\in e^*O(R,r^{-1}))(r^{-1}M\cap A\subseteq KM\cap A\subseteq U\cap A=\emptyset)
\end{array}$
\\
$\begin{array}{l}\Rightarrow (\exists M\in e^*O(R,y))(r^{-1}M\cap A=\emptyset)
\end{array}$
\\
$\begin{array}{l}\Rightarrow (\exists M\in e^*O(R,y))( M\cap rA=\emptyset)
\end{array}$
\\
 $\begin{array}{l}\Rightarrow y\notin e^*\text{-}cl(rA)
\end{array}$  
\\

Then, we have $rA\subseteq e^*\text{-}cl(rA)\subseteq rA$ which means $rA\in e^*C(R).$ 
\\

$b)$ It is proved similarly to $(a).$
\end{proof}

\begin{theorem}\label{t.4.5}Let $(R,+,\cdot,\tau)$ be an $e^*\text{-}$topological ring with unity and $A\subseteq R$. Then, the following properties hold:

$a)$ $r\cdot e^*\text{-}cl(A)\subseteq cl(rA)$ for each $r\in R,$

$b)$ $int(rA)\subseteq r\cdot e^*\text{-}int(A)$ for each $r\in R.$

$c)$ $r\cdot int(A)\subseteq e^*\text{-}int(A)$ for each $r\in R^*,$

$d)$ $e^*\text{-}cl(rA)\subseteq r\cdot cl(A)$ for each $r\in R^*,$

\end{theorem}
\begin{proof} $a)$ Let $x\in r\cdot e^*\text{-}cl(A)$ and $U\in O(R,x).$\\
$\left.\begin{array}{rr} x\in r\cdot e^*\text{-}cl(A) \Rightarrow   (\exists y\in e^*\text{-}cl(A))(x=ry)\\ U\in O(R,x)\end{array}\right\}\Rightarrow 
$
\\
$
\begin{array}{l}\Rightarrow (y\in e^*\text{-}cl(A))(\exists K\in e^*O(R,r))(\exists L\in e^*O(R,y))(K L\subseteq U)
\end{array}
$
\\
$
\begin{array}{l}\Rightarrow (\exists K\in e^*O(R,r))(\exists L\in e^*O(R,y))(K L\subseteq U)(L\cap A\ne\emptyset)
\end{array}
$
\\
$
\begin{array}{l}\Rightarrow \emptyset\ne KL\cap rA\subseteq U\cap rA
\end{array}
$
\\
$
\begin{array}{l}\Rightarrow \emptyset\ne U\cap rA
\end{array}
$

Then, we have $x\in cl(rA).$\\

$b)$ Let $x\in int(rA).$ Our aim is to show that $x\in r\cdot e^*\text{-}int(A).$\\
$\begin{array}{rcl} x\in int(rA) & \Rightarrow & (x\in rA)(int(rA)\in O(X,x)) \\ & {\Rightarrow} & (\exists y\in A)(x=ry)(int(rA)\in O(X,x)) \\ & {\Rightarrow} & (\exists U\in e^*O(R,r)) (\exists V\in e^*O(R,y))(rV\subseteq UV\subseteq int(rA)\subseteq rA) \\ & \Rightarrow & (\exists V\in e^*O(R,y))(V\subseteq A) \\ & \Rightarrow & y\in e^*\text{-}int(A) \\ & \Rightarrow &   x=ry\in r\cdot e^*\text{-}int(A).
\end{array}$

$c)$ Let $x\in r\cdot int(A).$
$$\begin{array}{rcl}
x\in r\cdot int(A) & \Rightarrow & r^{-1} x\in int(A) \\ & \Rightarrow & int(A)\in O(R,r^{-1}x) \\ & {\Rightarrow} & (\exists U\in e^*O(R,r^{-1})) (\exists V\in e^*O(R,x))(r^{-1} V\subseteq UV\subseteq int(A)\subseteq A) \\ & \Rightarrow & (\exists V\in e^*O(R,x))(V\subseteq rA) \\ & \Rightarrow & x\in e^*\text{-}int(rA).
\end{array}$$

$d)$ Let $x\in e^*\text{-}cl(rA)$ and $W\in O(R,r^{-1} x).$\\
$\left.\begin{array}{rr} W\in O(R,r^{-1} x) \Rightarrow (\exists U\in e^*O(R,r^{-1})) (\exists V\in e^*O(R,x))(r^{-1} V\subseteq UV\subseteq W)  \\ x\in e^*\text{-}cl(r A)\end{array}\right\}\Rightarrow 
$
\\
$\begin{array}{l}\Rightarrow (\exists U\in e^*O(R,r^{-1})) (\exists V\in e^*O(R,x))(r^{-1} V\subseteq UV\subseteq W)(V\cap rA\ne\emptyset)
\end{array}$
\\
$\begin{array}{l}\Rightarrow \emptyset\ne UV\cap A\subseteq W\cap A.
\end{array}$
\end{proof}
\begin{theorem}  Let $(R,+,\cdot,\tau)$ be an $e^*\text{-}$topological ring with unity and $r\in R^*$. Then, the function $f_r:R\to R$ defined by $f_r(x)=rx$  is  $e^*\text{-}$continuous.
\end{theorem}
\begin{proof} Let $U\in O(R).$ Our aim is to show that $f^{-1}[U]\in e^*O(R).$ For this, we will prove $f_r^{-1}[U]=e^*\text{-}int(f_r^{-1}[U]). $ We have always $e^*\text{-}int(f_r^{-1}[U])\subseteq f_r^{-1}[U]\ldots (1)$ 

Now, let $y\in f_r^{-1}[U].$\\
$\left. 
\begin{array}{r} 
y\in f_r^{-1}[U]=r^{-1}U \\ U\in O(R)\Rightarrow U=int(U)
\end{array} 
\right\} \Rightarrow \!\!\!\!
\begin{array}{c} 
\\  
\left. 
\begin{array}{r} 
y\in r^{-1}\cdot int(U)  \overset{\text{Theorem }\ref{t.4.5}}{\Rightarrow}  y\in e^*\text{-}int(r^{-1}U) \\ r^{-1} U=f_r^{-1}[U]
\end{array} 
\right\} \Rightarrow 
\end{array}$
\\
$\begin{array}{l}
\Rightarrow y\in f_r^{-1}[U] 
\end{array}$

Then, we have $f_r^{-1}[U]\subseteq e^*\text{-}int(f_r^{-1}[U])\ldots (2)$

$(1),(2)\Rightarrow f_r^{-1}[U]=e^*\text{-}int(f_r^{-1}[U])\Rightarrow f_r^{-1}[U]\in e^*O(X).$
\end{proof}

\begin{theorem} Let $(R,+,\cdot,\tau)$ be an $e^*\text{-}$topological ring and $A\subseteq R.$ Then, the following properties hold for each $x\in R.$

$a)$ $x+ e^*\text{-}cl(A)\subseteq cl(x+A),$ 

$b)$ $e^*\text{-}cl(x+A)\subseteq x+cl(A),$

$c)$ $x+int(A)\subseteq e^*\text{-}int(x+A),$

$d)$ $int(x+A)\subseteq x+ e^*\text{-}int(A).$
\end{theorem}

\begin{proof}
$a)$ Let $y\in x+e^*\text{-}cl(A).$ Our aim is to show that $y\in cl(x+A).$ Now, let $U\in O(R,y).$ If we prove $U\cap (x+A)\neq \emptyset$, then the proof complete.
\\
$\left.\begin{array}{rr} y\in x+e^*\text{-}cl(A) \Rightarrow   (\exists z\in e^*\text{-}cl(A))(y=x+z)\\ U\in O(R,y)\end{array}\right\}\Rightarrow 
$
\\
$
\begin{array}{l}\Rightarrow (\exists K\in e^*O(R,x))(\exists L\in e^*O(R,z))(\emptyset\ne (K+L)\cap (x+A)\subseteq U\cap (x+A))
\end{array}
$
\\
$
\begin{array}{l}\Rightarrow U\cap (x+A)\neq \emptyset.
\end{array}
$
\\

$b)$ Let $y\in e^*\text{-}cl(x+A).$ Our aim is to show that $y\in x+cl(A).$ Now, let $U\in O(R,-x+y).$
\\
$\begin{array}{rcl}
U\in O(R,-x+y)  \Rightarrow  (\exists K\in e^*O(R,-x) )(\exists L\in e^*O(R,y))(-x+L\subseteq K+L\subseteq U)
\end{array}$
\\
$\begin{array}{rcl}
\Rightarrow \emptyset\ne (-x+L)\cap A\subseteq U\cap A
\end{array}$

Therefore, $-x+y\in cl(A)$ which means $y\in x+cl(A).$\\

$c)$ Let $y\in x+ int(A).$ Our aim is to show that $y\in e^*\text{-}int(x+A).$\\
$\begin{array}{rcl} y\in x+int(A) & \Rightarrow & -x+y\in int(A)\in O(R) \\ & {\Rightarrow} & (\exists U\in e^*O(R,-x))(\exists V\in e^*O(R,y))(-x+V\subseteq U+V\subseteq int(A)\subseteq A) \\ & {\Rightarrow} & (\exists V\in e^*O(R,y))(V\subseteq x+A) \\  & \Rightarrow & y\in e^*\text{-}int(x+A). 
\end{array}$\\

$d)$ Let $y\in int(x+A).$ Our aim is to show that $y\in x+e^*\text{-}int(A).$
$$\begin{array}{rcl}
y\in int(x+A) & \Rightarrow & (\exists U\in O(R,y))(U\subseteq x+A) \\ & \Rightarrow & (\exists U\in O(R,y))(-x+U\subseteq A) \\ & \overset{\text{Theorem }\ref{t.3.4}}{\Rightarrow} & (-x+U\in e^*O(R,-x+y))(-x+U\subseteq A)
\\ & \Rightarrow & -x+y\in e^*\text{-}int(A)
\\ & \Rightarrow & y\in x+ e^*\text{-}int(A).\qedhere
\end{array}$$
\end{proof}

\begin{theorem} Let $(R,+,\cdot,\tau)$ 
be an $e^*\textsc{-}$topological ring and $A\subseteq R$. Then, we have the following properties.

$a)$ $-e^*\text{-}cl(A)\subseteq cl(-A)$,

$b)$ $e^*\text{-}cl(-A)\subseteq -cl(A)$,

$c)$ $-int(A)\subseteq e^*\text{-}int(-A)$,

$d)$ $int(-A)\subseteq -e^*\text{-}int(A)$.
\end{theorem}
\begin{proof}
$a)$ Let $y\notin cl(-A).$
$$\begin{array}{rcl} y\notin cl(-A) & \Rightarrow & (\exists U\in O(R,y))(U\cap (-A)=\emptyset) \\ & {\Rightarrow} & ( -U\in e^*O(R,-y))((-U)\cap A=\emptyset) \\ & {\Rightarrow} & -y\notin e^*\text{-}cl(A) \\ & {\Rightarrow} & y\notin -e^*\text{-}cl(A).\end{array}$$

$b)$ Let $y\notin -cl(A).$
$$\begin{array}{rcl} y\notin -cl(A) & \Rightarrow & -y\notin cl(A) \\ & \Rightarrow & (\exists U\in O(R,-y))(U\cap A=\emptyset) \\ & {\Rightarrow} & ( -U\in e^*O(R,y))((-U)\cap (-A)=\emptyset) \\ & {\Rightarrow} & y\notin e^*\text{-}cl(-A).\end{array}$$

$c)$ Let $y\in -int(A).$
$$\begin{array}{rcl} y\in -int(A) & \Rightarrow & -y\in int(A) \\ & {\Rightarrow} & (\exists U\in O(R,-y))(U\subseteq A) \\ & {\Rightarrow} & ( -U\in e^*O(R,y))(-U\subseteq -A) \\ & {\Rightarrow} & y\in e^*\text{-}int(-A).\end{array}$$

$d)$ Let $y\in int(-A).$
$$\begin{array}{rcl} y\in int(-A) & \Rightarrow & (\exists U\in O(R,y))(U\subseteq -A) \\ & {\Rightarrow} & (-U\in e^*O(R,-y))(-U\subseteq A) \\ & {\Rightarrow} & -y\in e^*\text{-}int(A) \\ & {\Rightarrow} & y\in -e^*\text{-}int(A).\qedhere \end{array}$$
\end{proof}

\begin{theorem}
Let $(R,+,\cdot,\tau)$ 
be an $e^*\textsc{-}$topological ring and $A\subseteq R.$ Then, we have the following properties for all $x\in R$. 

$a)$ $x+int(cl(\delta\text{-}int(A)))\subseteq cl(x+A),$

$b)$ $int(cl(\delta\text{-}int(x+A)))\subseteq x+cl(A),$

$c)$ $x+int(A)\subseteq cl(in(\delta\text{-}cl(x+A))),$

$d)$ $int(x+A)\subseteq x+cl(\delta\text{-}cl(A)).$
\end{theorem}

\begin{proof}
$a)$ Let $A\subseteq R$ and $x\in R.$
\\
$\begin{array}{l}
(A\subseteq R)(x\in R) \Rightarrow cl(x+A)\in C(R)\overset{\text{Theorem } \ref{t.3.4}}{\Rightarrow}   -x+cl(x+A)\in e^*C(R)
\end{array}$
\\
$\begin{array}{l}  \Rightarrow  int(cl(\delta\text{-}int(e^*\text{-}cl(A))))\subseteq int(cl(\delta\text{-}int(-x+cl(x+A))))\subseteq -x+cl(x+A)
\end{array}$
\\
$\begin{array}{l}  \Rightarrow   int(cl(\delta\text{-}int(A)))\subseteq int(cl(\delta\text{-}int(e^*\text{-}cl(A))))\subseteq -x+cl(x+A) 
\end{array}$
\\
$\begin{array}{l} \Rightarrow  x+int(cl(\delta\text{-}int(A)))\subseteq cl(x+A).
\end{array}$
\\

$b)$ Let $A\subseteq R$ and $x\in R.$\\
$\left.\begin{array}{rr} A\subseteq R \Rightarrow cl(A)\in C(R) \\ x\in R \end{array}\right\}\overset{\text{Theorem } \ref{t.3.4}}{\Rightarrow} x+cl(A)\in e^*C(R)$
\\
$
\begin{array}{l} \Rightarrow int(cl(\delta\text{-}int(x+A)))\subseteq int(cl(\delta\text{-}int(x+cl(A))))\subseteq x+cl(x+A).\end{array}
$
\\

$c)$ Let $A\subseteq R$ and $x\in R.$ \\
$
\left.\begin{array}{rr} A\subseteq R \Rightarrow int(A)\in O(R) \\ x\in R \end{array}\right\}\overset{\text{Theorem }\ref{t.3.4}}\Rightarrow x+int(A)\in e^*O(R)
$
\\
$
\begin{array}{l} \Rightarrow x+int(A)\subseteq cl(int(\delta\text{-}cl(x+int(A))))\subseteq cl(int(\delta\text{-}cl(x+A))).
\end{array}
$
\\

$d)$ Let $A\subseteq R$ and $x\in R.$\\
$\begin{array}{l} (A\subseteq R)(x\in R) \Rightarrow int(x+A)\in O(R) \overset{\text{Theorem }\ref{t.3.4} }{\Rightarrow} -x+int(x+A)\in e^*O(R)\end{array}$
\\
$\begin{array}{l} \Rightarrow  -x+int(x+A)\subseteq cl(int(\delta\text{-}cl(-x+int(x+A))))\subseteq cl(int(\delta\text{-}cl(A))).
\end{array}$
\end{proof}

\begin{theorem}
Let $(R,+,\cdot,\tau)$ be an $e^*$-topological ring and $A\subseteq R$. Then, we have the following properties.

$a)$ $-int(cl(\delta\text{-}int(A)))\subseteq cl(-A),$

$b)$ $int(cl(\delta\text{-}int(-A)))\subseteq -cl(A),$

$c)$ $-int(A)\subseteq cl(int(\delta\text{-}cl(-A))),$

$d)$ $int(-A)\subseteq -cl(int(\delta\text{-}cl(A))).$
\end{theorem}

\begin{proof} $a)$ Let $A\subseteq R.$
$$\begin{array}{rcl} A\subseteq R & \Rightarrow & cl(-A)\in C(R) \\ & \overset{\text{Theorem }\ref{t.3.4} }{\Rightarrow} & -cl(-A)\in e^*C(R) \\ & {\Rightarrow} &  int(cl(\delta\text{-}int(A)))\subseteq int(cl(\delta\text{-}int(-cl(-A))))\subseteq -cl(-A)  \\  & \Rightarrow & -int(cl(\delta\text{-}int(A)))\subseteq cl(-A). 
\end{array}$$

$b)$ Let $A\subseteq R.$
$$\begin{array}{rcl} A\subseteq R & \Rightarrow & cl(A)\in C(R) \\ & \overset{\text{Theorem }\ref{t.3.4} }{\Rightarrow} & -cl(A)\in e^*C(R) \\ & {\Rightarrow} &  int(cl(\delta\text{-}int(-A)))\subseteq int(cl(\delta\text{-}int(-cl(A))))\subseteq -cl(A).
\end{array}$$

$c)$ Let $A\subseteq R.$
$$\begin{array}{rcl} A\subseteq R & \Rightarrow & int(A)\in O(R) \\ & \overset{\text{Theorem }\ref{t.3.4} }{\Rightarrow} & -int(A)\in e^*O(R) \\ & {\Rightarrow} &  -int(A)\subseteq cl(int(\delta\text{-}cl(-int(A))))\subseteq cl(int(\delta\text{-}cl(-A))). 
\end{array}$$

$d)$ Let $A\subseteq R.$
$$\begin{array}{rcl} A\subseteq R & \Rightarrow & int(-A)\in O(R) \\ & \overset{\text{Theorem }\ref{t.3.4} }{\Rightarrow} & -int(-A)\in e^*O(R) \\ & {\Rightarrow} &  -int(-A)\subseteq cl(int(\delta\text{-}cl(-int(-A))))\subseteq cl(int(\delta\text{-}cl(A))) \\ & {\Rightarrow} & int(-A)\subseteq -cl(int(\delta\text{-}cl(A))).\qedhere
\end{array}$$
\end{proof}
\begin{theorem}\label{t411}Let $(R,+,\cdot,\tau)$ be an $e^*$-topological ring and $A,B\subseteq R$. Then, $e^*\text{-}cl(A)+e^*\text{-}cl(B)\subseteq cl(A+B).$
\end{theorem}
\begin{proof} Let $z\in e^*\text{-}cl(A)+e^*\text{-}cl(A).$ Our aim is to show that $z\in cl(A+B).$ Now, let $W\in O(R,z).$
\\
$\left.\begin{array}{rr} z\in e^*\text{-}cl(A)+e^*\text{-}cl(A)\Rightarrow (\exists x\in e^*\text{-}cl(A))(\exists y\in e^*\text{-}cl(A))(z=x+y) \\ W\in O(R,z)\end{array}\right\}\Rightarrow $
\\
$\begin{array}{l}\Rightarrow (\exists U\in O(R,x))(\exists V\in O(R,y))(U+V\subseteq W)(U\cap A\neq\emptyset)(V\cap B\neq\emptyset)\end{array}$
\\
$\begin{array}{l}\Rightarrow (W\in O(R,z))((U\cap A)+(V\cap B)\ne\emptyset)(U+V\subseteq W)\end{array}$
\\
$\begin{array}{l}\Rightarrow (W\in O(R,z))(\exists t\in R)(t\in (U\cap A)+(V\cap B))(U+V\subseteq W)\end{array}$
\\
$\begin{array}{l}\Rightarrow (W\in O(R,z))(\exists u\in U\cap A)(\exists v\in V\cap B)(t=u+v)(U+V\subseteq W)\end{array}$
\\
$\begin{array}{l}\Rightarrow (W\in O(R,z))(u\in U)(u\in A)(v\in V)(v\in B)(U+V\subseteq W)\end{array}$
\\
$\begin{array}{l}\Rightarrow (W\in O(R,z))(u+v\in U+V)(u+v\in A+B)(U+V\subseteq W)\end{array}$
\\
$\begin{array}{l}\Rightarrow (W\in O(R,z))(u+v\in (U+V)\cap (A+B)\subseteq W\cap (A+B))\end{array}$
\\
$\begin{array}{l}\Rightarrow (W\in O(R,z))(W\cap (A+B)\neq \emptyset).\end{array}$
\end{proof}

\begin{remark}
The converse of inclusion given in Theorem \ref{t411} need not be true as shown by the following example.    
\end{remark}
\begin{example} 
Let $R=\{a,b,c,d\}$ and  $\tau=\{\emptyset,R,\{a\},\{a,b\}\}.$ Let the addition and multiplication operations on $R$ be as given in Example \ref{34}. For the subsets $A=\{a\}$ and $B=\{c\}$, we have $cl(A+B)=cl(\{c\})=\{c,d\}$ and $e^*\text{-}cl(A)+e^*\text{-}cl(B)=e^*\text{-}cl(\{a\})+e^*\text{-}cl(\{c\})=\{a\}+\{c\}=\{c\}.$ It is obvious that 
$cl(\{c\})=\{c,d\}\nsubseteq \{c\}= e^*\text{-}cl(A)+e^*\text{-}cl(B).$
\end{example}
\begin{theorem}Let $(R,+,\cdot,\tau_1)$ be an $e^*$-topological ring and
let $(S,+,\cdot,\tau_2)$ be a topological ring. If a ring homomorphism $f: R\to S$ is continuous at $0_R$, then $f$ is $e^*\text{-}$continuous.
\end{theorem}
\begin{proof} Let $f$ be a homomorphism and continuous at $0_R.$ Our aim is to show that $f$ is $e^*$-continuous. Now, let $V\in O(R,f(x)).$\\
$\left.\begin{array}{rr} V\in O(R,f(x)) \Rightarrow f(x)=f(x+0_R)\in V\in O(R) \\ f\text{ is homomorphism}\end{array}\right\}\Rightarrow 
$
\\
$\left.\begin{array}{rr} \Rightarrow f(x)+f(0_R)\in V\in O(R) \Rightarrow -f(x)+V\in O(R,f(0_R)) \\ f\text{ is continuous in} \ 0_R \end{array}\right\}\Rightarrow 
$
\\
$\begin{array}{l}\Rightarrow (\exists W\in O(R, O_R))(f[W]\subseteq -f(x)+V)\end{array}$
\\
$\left.\begin{array}{rr} \Rightarrow  (\exists W\in O(R, O_R))(f(x)+f[W]\subseteq V) \\ f \text{ is homomorphism}   \end{array}\right\}\Rightarrow $
\\
$\left.\begin{array}{rr} \Rightarrow  (\exists W\in O(R, O_R))(f[x+W]\subseteq V)  \\ U:=x+W \end{array}\right\}\Rightarrow  ( U\in e^*O(R,x)) (f[U]\subseteq V).
$
\end{proof}
\section{Acknowledgements}
We would like to thank the anonymous reviewer(s) for their careful reading of our manuscript and their insightful comments and suggestions.

\bibliographystyle{amsplain}

\end{document}